\documentclass[12pt]{amsart}%
\usepackage{amsfonts}
\usepackage{latexsym}
\usepackage{amssymb}
\usepackage{amsmath}
\usepackage{amstext}
\usepackage{amsthm}
\usepackage{amsxtra}
\usepackage[utf8]{inputenc}
\usepackage[brazilian,english]{babel}
\providecommand{\U}[1]{\protect\rule{.1in}{.1in}}
\newtheorem{theorem}{Theorem}

\newtheorem{corollary}{Corollary}
\newtheorem{example}{Example}
\newtheorem{proposition}{Proposition}
\newtheorem{remark}{Remark}
\theoremstyle{definition}

\usepackage{enumerate}

\def\({\left(}       \def\){\right)}

\begin{document}

\title{Lower and upper order of harmonic mappings}

\author[H. Arbel\'aez]{Hugo Arbel\'{a}ez}
\address{Escuela de Matem\'aticas, Universidad Nacional de Colombia, Medell\'{\i}n, Colombia}
\email{hjarbela@unal.edu.co}

\author[R. Hern\'andez]{Rodrigo Hern\'andez}
\address{Facultad de Ingenier\'{\i}a y Ciencias, Universidad Adolfo Ib\'a\~nez. Av. Pa\-dre Hurtado 750. Vi\~na del Mar, Chile.} \email{rodrigo.hernandez@uai.cl}

\author[W. Sierra]{Willy Sierra}
\address{Departamento de Matem\'aticas, Universidad del Cauca, Popay\'{a}n, Colombia}
\email{wsierra@unicauca.edu.co}



\thanks{The first author was partially supported by the Universidad Nacional de Colombia, Hermes code 49148. The second author is supported by grant Fondecyt $\#1190756$, Chile. The third author thanks the Universidad del Cauca for providing time for this work through research project VRI ID 5464.}

\maketitle

\begin{abstract}
In this paper, we define both the upper and lower order of a sense-preserving harmonic mapping in $\mathbb{D}$. We generalize to the harmonic case some known results about holomorphic functions with positive lower order and we show some consequences of a function having finite upper order. In addition, we improve a related result in the case when there is equality in a known distortion theorem for harmonic mappings with finite upper order. Some examples are provided to illustrate the developed theory.
\end{abstract}



\section{Introduction}
In 1964 Pommerenke \cite{Po64} introduces the linear invariance order for a locally univalent holomorphic function defined in the unit disk $\mathbb{D}:=\left\lbrace z: |z|<1 \right\rbrace $ by
\begin{equation*}
\alpha(f):=\sup_{z\in\mathbb{D}}|A_f(z)|,
\end{equation*}
where, as in \cite{CrPo06,Po08},
\begin{equation}\label{Af1}
A_f(z)=\frac{1-|z|^2}{2}\frac{f''(z)}{f'(z)}-\overline{z}.
\end{equation}
This operator has played an important role in the study of both analytic and geometric properties of locally univalent holomorphic functions; it is closely related to (Euclidean) convexity and concavity properties of the function under study, as well as to growth and distortion results, some of which can be obtained  by using the connection between $A_f,$ the second coefficient $a_2=f''(0)/2$ and the notion of linear invariant family of locally univalent analytic functions. An important result in this direction was proved by Pommerenke \cite{Po64}, who showed that $\alpha(f)\geq1$ for all locally univalent holomorphic function $f$ and $\alpha(f)=1$ exactly if $f$ is a convex univalent function. Likewise, a straightforward calculation shows that (see for example \cite{Po08})
\begin{equation}\label{Af analytic and Koebe}
\frac{f\left(\frac{z+z_0}{1+\overline{z_0}z} \right)-f(z_0) }{(1-|z_0|^2)f'(z_0)}=z+A_f(z_0)z^2+\cdots,\quad z\in\mathbb{D},
\end{equation}
for all $z_0\in\mathbb{D}$ given. Also,
\begin{equation}\label{Af analytic and poincare metric}
\frac{\partial}{\partial z}\log|(1-|z|^2)f'(z)|=\frac{1}{1-|z|^2}A_f(z),\quad z\in\mathbb{D}
\end{equation}
and, if $\varphi:\mathbb{D}\to\mathbb{D}$ is analytic and locally univalent, then
\begin{equation*}
A_{f\circ\varphi}(z)=\frac{(1-|z|^2)\varphi'(z)}{1-|\varphi(z)|^2}\left( A_f(\varphi(z))+\overline{\varphi(z)} \right)+A_{\varphi}(z), 
\end{equation*}
for all $z\in\mathbb{D}.$ In particular, if $\varphi\in \text{Aut}(\mathbb{D}),$ where $\text{Aut}(\mathbb{D})$ denotes the family of automorphisms of $\mathbb{D},$ then
\begin{equation}\label{chain rule Af analytic}
A_{f\circ\varphi}(z)=\frac{\varphi'(z)}{|\varphi'(z)|}A_f(\varphi(z)), 
\end{equation}
for all $z\in\mathbb{D}.$

Due to the previously described, and based on the work developed by Pommerenke, some authors have defined and studied in other contexts operators with similar properties to $A_f$. For example, Ma and Minda \cite{MaMi92,MaMi94} investigated topics related to linear invariance but in two non Euclidean geometries: they proposed both a definition of spherical linear invariance for locally univalent meromorphic functions defined on the unit disk and a definition of hyperbolic linear invariance for locally univalent functions that map the unit disk into itself. More precisely, they defined the order of a locally univalent meromorphic function $f$ defined in $\mathbb{D}$ by $\alpha_s(f):=\sup_{z\in\mathbb{D}}|A^{\#}_f(z)|,$ where
\begin{equation*}
A^{\#}_f(z)=\frac{1-|z|^2}{2}\frac{f''(z)}{f'(z)}-\overline{z}-\frac{(1-|z|^2)f'(z)\overline{f(z)}}{1+|f(z)|^2},
\end{equation*}
and the order of a locally univalent holomorphic function $f:\mathbb{D}\to \mathbb{D}$ by
$\alpha_h(f):=\sup_{z\in\mathbb{D}}|A^{h}_f(z)|,$ where 
\begin{equation*}
A^{h}_f(z)=\frac{1-|z|^2}{2}\frac{f''(z)}{f'(z)}-\overline{z}+\frac{(1-|z|^2)f'(z)\overline{f(z)}}{1-|f(z)|^2}.
\end{equation*}
Just like in the Euclidean setting, in both cases the authors proved that the order always is greater or equal than $1$, and the equality occurs precisely when $f$ is spherically or hyperbolically convex, respectively. We remark that other authors, in particular J. A. Pfaltzgraff and T. J. Suffridge \cite{97 linear invariance in Cn,99 linear invariance in Cn}, have investigated the $n$-dimensional version of Pommerenke's theory of (Euclidean) linear invariance. It has also been extended to the setting of planar harmonic mappings, scenario in which the present paper is framed. In this context the pioneering work was that of T. Sheil-Small \cite{Sheil small}, in which introduced  the notion of an affine and linear invariant family of univalent harmonic functions.

On the other hand, Cruz and Pommerenke \cite{CrPo06} defined the concept of lower linear-invariance order for a locally univalent holomorphic function defined in $\mathbb{D}$ by
\begin{equation*}
\mu(f):=\inf_{z\in\mathbb{D}}|A_f(z)|.
\end{equation*}
It is known that $0\leq\mu(f)\leq1$ \cite[Proposition 2]{CrPo06}. Moreover, the lower order is closely related with properties of concavity; for example, in \cite{CrPo06} is proved that $0\leq \mu(f)\leq 1$ and if $f$ is univalent, $\mu(f)=1$ holds exactly if $f$ is concave. Subsequently in \cite{Po08} the author carries out a deeper study of the lower order and proved interesting results under the hypothesis $\mu(f)>0$. More recently and following the ideas of Pommerenke, J. Arango \textit{et al.} in \cite{AAM17} and \cite{AAM19}, defined the lower linear invariant spherical order by $\mu_s(f)=\inf_{z\in\mathbb{D}}|A^{\#}_f(z)|$ and the lower linear invariant hyperbolic order by $\mu_h(f)=\inf_{z\in\mathbb{D}}|A^{h}_f(z)|$. In  \cite{AAM17} and \cite{AAM19}, the authors noted that, as in the Euclidean case, $0\leq\mu_s(f)\leq 1$ and there are functions with lower order positive, however, in the hyperbolic case, every function $f$, satisfies that $\mu_h(f)=0$.

In this work, in a similar way to the ideas discussed above, we introduce the concepts of lower and upper order for a sense-preserving harmonic mapping $f:\mathbb{D}\to \mathbb{C}$, which match respectively with $\mu(f)$ and $\alpha(f)$ in the case when $f$ is analytic. Although these notions appear implicitly in previous researches about affine and linear invariant families of harmonic functions, we focus our attention on properties of the operator $A_f$ in the setting of planar harmonic mappings; in Section\,\ref{sec Af} we show that it has invariance properties similar to those of the operator defined by (\ref{Af1}). In Section\,\ref{section lower order}, we define the lower order for a sense-preserving harmonic mapping, we show that there are functions with positive lower order, and we generalize, to the harmonic case, some of the results obtained by Pommerenke in \cite{Po08}. In Section\,\ref{section linear invariance}, we establish a relation between the upper order and linearly connected domains. Finally, we study a characterization of functions with finite order in terms of bounds for the Jacobian.  In addition, we exhibit some interesting examples, which illustrate the theory.

\section{Preliminaries on harmonic mappings}\label{Section preliminaries}

In this section we introduce briefly some basic results about harmonic functions in the plane, which are used throughout this paper.

Let $f$ be a planar harmonic mapping defined on a domain $\Omega\subseteq\mathbb{C}.$ It is well known that if $\Omega$ is simply connected, $f$ has the canonical representation $f=h+\overline{g}$, where $h$ and $g$ are analytic functions in $\Omega$; this representation is unique up to an additive constant, which
is usually determined by imposing the condition $g(z_0)=0$ for some $z_0$ fixed in $\Omega$. Lewy \cite{Le} proved that $f$ is locally univalent in $\Omega$ if and only if its Jacobian $J_f=|h'|^2-|g'|^2$ does not vanish. Thus, if $f$ is locally univalent in $\Omega,$ it is either sense-preserving or sense-reversing depending on the conditions $J_f>0$ or $J_f<0$ throughout the domain $\Omega,$ respectively.  Along this paper we will consider sense-preserving harmonic mappings on the unit disk $\mathbb{D},$ in this case the analytic part $h$ is locally univalent in $\mathbb{D}$ and the second complex dilatation of $f,$ $\omega = g'/h',$ is an analytic function in $\mathbb{D}$ satisfying $|\omega|<1.$ The family of all sense-preserving univalent harmonic mappings $f=h+\overline{g}$ defined in $\mathbb{D}$, normalized by $h(0) = 0,$ $g(0) = 0,$ and $h'(0) = 1$, will be denoted by $S_H$. Also, $S_H^0$ will denote the subclass of functions in $S_H$ that satisfy the further normalization $g'(0) = 0,$ $K_H$ will denote the subclass of functions in $S_H$ that are convex, and $K^0_H=K_H\cap S_H^0.$

\section{The operator $A_f.$}\label{sec Af}

In a similar form to the analytic case, given a sense-preserving harmonic mapping $f=h+\overline{g}$ in $\mathbb{D},$ we define the operator $A_{f}$ by
\[
A_{f}(z)  =\frac{\left(1-|z|^{2}\right)
}{2}P_{f}(z)-\overline{z},\qquad z\in\mathbb{D},
\]
where
\begin{equation}\label{preschwarzian}
	P_{f}=\partial_z\log J_{f} =\frac{h''}{h'}-\frac{\overline{\omega}\omega'}{1-\left\vert
		\omega\right\vert ^{2}}
\end{equation}
is the \textit{pre-schwarzian derivative of}$f$ introduced in \cite{HM}. Hence, we can express $A_f$ in the form
\begin{equation}\label{A_f harmonic 2}
	A_f(z)=A_h(z)-\frac{(1-|z|^2)}{2}\frac{\overline{\omega(z)}\omega'(z)}{1-\left\vert
		\omega(z)\right\vert ^{2}},\qquad z\in\mathbb{D},
\end{equation}
where $A_h$ is given by \eqref{Af1}. Observe that if $f$ is analytic, then $\omega=0$ and therefore $A_f=A_h.$ Thus, in the analytic case $A_f$ coincides with the operator introduced in \cite{Po64} for analytic functions defined in $\mathbb{D}.$

We remark that $A_f$ arises in a natural way when studying problems related to the second coefficient of affine and linear invariant families of harmonic mappings in the unit disk. Indeed, as in the analytic case, it appears in an expression of the type \eqref{Af analytic and Koebe}, which can be seen by the following standard argument: given $a\in\mathbb{D},$ we consider the function
\[F(z)=\frac{f\left(\frac{a+z}{1+\overline{a}z} \right) -f(a)}{(1-|a|^2)h'(a)}:=H(z)+\overline{G(z)},\]
which satisfies $A_f(a)=A_F(0).$  So, if $B_1:=G'(0),$
\[F_0(z)=\frac{F(z)-\overline{B_1F(z)}}{1-|B_1|^2}=H_0(z)+\overline{G_0(z)}\]
satisfies $H_0'(0)=1,$ $G_0'(0)=0,$ and
\begin{equation}\label{eq Af and a2}
A_f(a)=A_F(0)=A_{F_0}(0)=A_{H_0}(0)=\frac{H_0''(0)}{2},
\end{equation}
being the penultimate equality a consequence of \eqref{A_f harmonic 2}.

Further properties of $A_f,$ including among them equalities of the type \eqref{Af analytic and poincare metric} and \eqref{chain rule Af analytic},  are provided in the following proposition. The proof is a straightforward calculation, which we include here in order to make our exposition self contained.

\begin{proposition}\label{prop properties of A_f}
	Let $f=h+\overline{g}$ be a sense-preserving harmonic mapping in $\mathbb{D}$. Then
	
	(i) $A_f$ does not vanish identically;
	
	(ii) For all $\sigma\in \text{Aut}\left( \mathbb{D}\right),$ we have $A_{f\circ\sigma}=\dfrac{\sigma'}{\left\vert \sigma'\right\vert} \left(  A_{f}\circ\sigma\right);$
	
	(iii) For all affine mapping $L\left(  z\right)
	=az+b\overline{z}+c,$ with $a,b,c\in\mathbb{C}$ and $\left\vert b/a\right\vert
	<1,$ we have $A_{L\circ f}=A_{f}.$ 
\end{proposition}
\begin{proof}
	To prove \textit{(i)}, we first note that
	\begin{equation}\label{A1}
		J_{f}=\left\vert h'\right\vert ^{2}-\left\vert g'\right\vert
		^{2}\leq\left\vert h'\right\vert ^{2}
	\end{equation}
	and moreover $A_{f}$ can be expressed in the form
	\begin{equation}\label{diferential form of A_f}
		A_{f}(z)=(1-|z|^{2})\frac{\partial}{\partial z}\log\left\{  (1-|z|^{2}%
		)J_{f}^{1/2}(z)\right\}.
	\end{equation}
	Thus, if $A_{f}$ is identically zero, then there is $k\in\mathbb{R}$
	such that
	\[
	(1-|z|^{2})J_{f}^{1/2}(z)=k
	\]
	for all $z\in\mathbb{D}.$ It follows from here and from (\ref{A1}) that 
	\[
	\left\vert h'(z)\right\vert \geq\frac{k}{1-|z|^{2}},\qquad\text{ for all }z\in\mathbb{D},
	\]
	from where $\left\vert h'\right\vert \rightarrow\infty$ as $|z|\rightarrow1,$ which is a contradiction with the maximum principle for the analytic function $1/h'.$
	
	Now we prove \textit{(ii)}. A straightforward calculation shows that
	\[
	P_{f\circ\sigma}=\left(  \frac{h''}{h'}\circ
	\sigma\right)  \sigma'+\frac{\sigma''}{\sigma'}%
	-\frac{\overline{\omega\circ\sigma}\left(  \omega\circ\sigma\right)'}{1-\left\vert \omega\circ\sigma\right\vert ^{2}}\sigma'=\left(  P_{f}\circ\sigma\right)  \sigma'+\frac{\sigma''}{\sigma'}.
	\]
	Thus, by Schwarz-Pick's lemma,
	\begin{align*}
		A_{f\circ\sigma}(z)  &  =\frac{\left(  1-\left\vert z\right\vert
			^{2}\right)  \sigma'(z)  }{2}P_{f}\left(  \sigma(z) \right)  +\frac{\left(  1-\left\vert z\right\vert ^{2}\right)  }{2}\frac{\sigma''}{\sigma'}(z)-\overline{z}\\
		&  =\frac{\left(  1-\left\vert \sigma(z)\right\vert^{2}\right)}{2}\frac{\sigma'(z)}{\left\vert
			\sigma'(z)\right\vert }P_{f}\left(  \sigma(z)\right)  -\frac{\sigma'(z)}{\left\vert\sigma'(z)\right\vert}\overline{\sigma(z)}\\
		&=\frac{\sigma'(z)}{\left\vert \sigma'(z)\right\vert}A_{f}\left(\sigma(z)\right),
	\end{align*}
	for $z\in\mathbb{D}$ arbitrary.
	
	Finally, by Proposition\,1 in \cite{HM} we have $P_{L\circ f}=P_{f}.$ So,
	\[
	A_{L\circ f}\left(  z\right)  =\frac{\left(  1-\left\vert z\right\vert
		^{2}\right)  }{2}P_{L\circ f}\left(  z\right)  -\overline{z}=\frac{\left(
		1-\left\vert z\right\vert ^{2}\right)  }{2}P_{f}\left(  z\right)
	-\overline{z}=A_{f}\left(  z\right),
	\]
	for $z\in\mathbb{D}$ arbitrary, which proves \textit{(iii)}.
\end{proof}

\section{Lower order}\label{section lower order}

Following the ideas presented in \cite{CrPo06} and \cite{Po08}, we introduce the notion of lower order in the setting of harmonic mappings and we will prove similar results to Proposition 4.1 and Proposition 5.1 in \cite{Po08}.

Let $f=h+\bar{g}$ be a sense-preserving harmonic mapping in $\mathbb{D}$. We define the \textit{lower order} of $f$ by
\[\mu(f)=\inf_{z\in \mathbb{D}}|A_f(z)|,\]
where 
\[A_f(z)=A_h(z)-\frac{(1-|z|^2)}{2}\frac{\overline{\omega(z)}\omega'(z)}{1-|\omega(z)|^2},\qquad z\in\mathbb{D}.\]
From the properties of $A_f$ it follows that $\mu(L\circ f\circ \sigma)=\mu(f)$, for $L$ an affine mapping and $\sigma \in \text{Aut}(\mathbb{D})$. Since $0\leq \mu(h)\leq 1,$ because $h$ is a locally univalent analytic function, we obtain from the Schwarz-Pick classical inequality that $0\leq\mu(f)\leq 3/2.$ To illustrate we present two examples, which show that the bound 3/2 is sharp, and that there are harmonic mappings with positive lower order.
\begin{example}\label{example L} We consider the harmonic mapping
	\[
	L(z)  =\frac{1}{2}[l(z)+k(z)]+\overline{\frac{1}{2}[l(z)-k(z)]},
	\]
	where
\[l(z)=\frac{z}{1-z}\qquad\text{and}\qquad k(z)=\frac{z}{(1-z)^{2}}.\]
We know that $L$ belongs to $K_H^0$ and it maps $\mathbb{D}$ onto the full half plane $\text{Re}\left\lbrace w\right\rbrace >-1/2,$ see for example \cite[p. 41]{Du04}. Now, a direct calculation gives
\[A_L(z)=\frac{3}{2}\left( \frac{1-\bar{z}}{1-z}\right),\qquad z\in\mathbb{D},\]
and consequently $\mu(L)=3/2,$ which shows that the bound referred above is sharp.
\end{example}
\begin{example} \label{example koebe}
We now consider the Koebe harmonic function, which is defined by $K(z)=h(z)+\overline{g(z)}=k(z)+2\text{Re}\left\lbrace g(z) \right\rbrace ,$ where
\[h(z)=\frac{z-z^2/2+z^3/6}{(1-z)^3}\qquad\text{ and }\qquad g(z)=\frac{z^2/2+z^3/6}{(1-z)^3}.\]
The Koebe harmonic function was constructed by Clunie and Sheil-Small \cite{Clunie and Sheil} as a candidate to play the role of extremal function for some problems in the class $S_H^0,$ see also \cite[p. 82]{Du04}. It maps $\mathbb{D}$ harmonically onto the full plane slit along the negative real axis from $-1/6$ to infinity.

After an algebraic calculation, we obtain
\[A_K(z)=\frac{3}{2}\left[ \frac{1-|z|^2}{1-z^2}\left(\frac{5}{3}+z \right)-\overline{z}  \right]=\frac32\left[\frac{\frac53(1-|z|^2)+2i\rm{Im}\{z\}}{1-z^2}\right] .\] A straightforward calculation gives that $$\left|\frac53(1-|z|^2)+2i\rm{Im}\{z\}\right|^2-\left|1-z^2\right|^2=\frac{16}{9}\left(1-|z|^2\right)^2.$$ Then, we have that $$|A_K(z)|=\frac32+2\cdot\frac{1-|z|^2}{|1-z^2|}.$$ Therefore, $\mu(K)=3/2.$
\end{example}
We recall that a continuous, injective, and  unbounded function $f:\mathbb{D}\to\mathbb{C}$ is said to be concave, if $\mathbb{C}\setminus f(\mathbb{D})$ is convex. The functions $L$ and $K$ of Examples \ref{example L} and \ref{example koebe} are concave harmonic functions. As was mentioned in the introduction, if $f$ is analytic and locally univalent in $\mathbb{D}$, then $0\leq \mu(f)\leq 1.$ When $f$ is univalent, the upper bound is attained if and only if $f$ is concave, which is equivalent (if we assume $f(\mathbb{D})$ with angle $\pi \alpha$ at $\infty,$ $1\leq \alpha\leq 2$) to that $f$ satisfies
\[\text{Re}\left\lbrace \frac{\alpha+1}{2}\frac{1+z}{1-z}-1-z\frac{f''(z)}{f'(z)} \right\rbrace >0,\qquad z\in\mathbb{D},\]
see \cite[p. 154]{CrPo06}. In this direction, in the setting of harmonic functions, we provide a lower bound for the lower order of a sense-preserving harmonic mapping $f=h+\overline{g},$ satisfying the property $\varphi_{\lambda}=h+\lambda g$ concave for all $|\lambda|=1$. The existence of such functions can be proved as follows. For any $\alpha>1$ there existe an unbounded univalent analytic function $h$ satisfying
\[\text{Re}\left\lbrace \frac{\alpha+1}{2}\frac{1+z}{1-z}-1-z\frac{h''(z)}{h'(z)} \right\rbrace\geq \beta>0,\qquad z\in\mathbb{D}, \]
for some $1<\alpha\leq 2$ (for instance $h(z)=(1-z)^{-(1+2\beta)}$). Therefore,
\begin{align*}
\text{Re}\left\lbrace \frac{\alpha+1}{2}\frac{1+z}{1-z}-1-z\frac{\varphi_{\lambda}''(z)}{\varphi_{\lambda}'(z)} \right\rbrace&= \text{Re}\left\lbrace \frac{\alpha+1}{2}\frac{1+z}{1-z}-1-z\left(\frac{h''(z)}{h'(z)}+\frac{\lambda \omega'(z)}{1+\lambda\omega(z)} \right)  \right\rbrace\\
& = \text{Re}\left\lbrace \frac{\alpha+1}{2}\frac{1+z}{1-z}-1-z\frac{h''(z)}{h'(z)} \right\rbrace-\text{Re}\left\lbrace \frac{\lambda z \omega'(z)}{1+\lambda\omega(z)} \right\rbrace\\
&\geq \beta- \text{Re}\left\lbrace \frac{\lambda z \omega'(z)}{1+\lambda\omega(z)} \right\rbrace.
\end{align*}
In consequence, if we put the condition that $\omega(z)=\rho z,$ with $|\rho|<1,$ we obtain that
\[\text{Re}\left\lbrace \frac{\alpha+1}{2}\frac{1+z}{1-z}-1-z\frac{\varphi_{\lambda}''(z)}{\varphi_{\lambda}'(z)} \right\rbrace\geq \beta- \text{Re}\left\lbrace \frac{\rho\lambda z}{1+\rho\lambda z} \right\rbrace\geq \beta-\frac{|\rho|}{1-|\rho|}.\]
It follows that $\varphi_{\lambda}$ is concave for all $|\lambda|=1,$ if $|\rho|<\beta/(1+\beta),$ \cite[p. 154]{CrPo06}.
\begin{proposition}\label{prop stable concave}
Let $f=h+\overline{g}:\mathbb{D}\to\mathbb{C}$ be a  sense-preserving harmonic mapping such that $\varphi_\lambda=h+\lambda g$ is concave for all $|\lambda|=1$. Then $1\leq \mu(f)\leq 3/2.$
\end{proposition}
\begin{proof}
It only remains to prove that $1\leq \mu(f).$ A tedious but straightforward calculation yields
\begin{equation}\label{eq aux shc}
	A_f(z)=A_{\varphi_\lambda}(z)-\frac{\lambda+\overline{\omega(z)}}{1+\lambda\omega(z)}\frac{(1-|z|^2)\omega'(z)}{2(1-|\omega(z)|^2)},
\end{equation}
for all $z\in\mathbb{D}$ and $|\lambda|=1.$ From the concavity of $\varphi_{\lambda}$ we know that $\mu(\varphi_{\lambda})=1,$ so
\[\left|A_f(z) \right|\geq \left|  A_{\varphi_\lambda}(z)\right| -\left| \frac{\lambda+\overline{\omega(z)}}{1+\lambda\omega(z)}\frac{(1-|z|^2)\omega'(z)}{2(1-|\omega(z)|^2)}\right| \geq 1- \frac{1}{2}\left| \frac{(1-|z|^2)\omega'(z)}{1-|\omega(z)|^2}  \right|\geq  \frac{1}{2}.\]
On the other hand, \eqref{eq aux shc} can be expressed in the form 
\[A_f(z)=A_{\varphi_\lambda}(z)-\sigma(\lambda)\frac{(1-|z|^2)\omega'(z)}{2(1-|\omega(z)|^2)},\]
for all $z\in\mathbb{D}$ and $|\lambda|=1,$ where $\sigma$ is the automorphism of $\mathbb{D}$ defined by
\begin{equation}\label{eq def auth sigma}
\sigma(\zeta)=\frac{\zeta+\overline{\omega(z)}}{1+\zeta\omega(z)}.
\end{equation}
Thus, for each $z\in\mathbb{D}$ we can choose $\lambda,$ with $|\lambda|=1,$ such that
\[ 1\leq \left| A_{\varphi_\lambda}(z) \right|=\left| A_f(z) \right|- \frac{1}{2}\left| \frac{(1-|z|^2)\omega'(z)}{(1-|\omega(z)|^2)}  \right|\leq \left| A_f(z) \right|,  \]
which completes the proof.  
\end{proof}

Next, we present a criterion to establish a positive lower bound of $\mu(f),$ which is obtained as follows (see Proposition\,3.1 in \cite{Po08} for the analytic case): we suppose that $f=h+\bar{g}$ is a sense-preserving harmonic mapping in $\mathbb{D},$ $0\leq \lambda \leq 1,$ and
\begin{equation*}
	\left|P_f(z)-\frac{2}{1-z}\right| \leq \frac{2\lambda}{1-|z|^2},\qquad z\in\mathbb{D}.
\end{equation*}
Then, from the equality
\begin{equation*}
	A_f(z)-\frac{1-\overline{z}}{1-z}=\frac{1-|z|^2}{2}P_f(z)-\frac{1-|z|^2}{1-z}=\frac{1-|z|^2}{2}\left[P_f(z)-\frac{2}{1-z} \right], 
\end{equation*}
it follows that
\begin{equation*}\label{eq. aux lower bound u_f}
	\left| A_f(z)\right| \geq 1-\lambda,
\end{equation*}
for all $z\in \mathbb{D},$ from which we deduce that $\mu(f)\geq 1-\lambda.$ As a consequence, we have the following example.
\begin{example}
	Let $f$ be the harmonic mapping defined by
	\[f(z)=\dfrac{z}{1-z}-\overline{\left(\dfrac{z}{1-z}+\log(1-z) \right) }.\]
	By direct calculation one sees that
	\[P_f(z)=\frac{2}{1-z}-\frac{\overline{z}}{1-|z|^2},\]
	whence
	\[\left|P_f(z)-\frac{2}{1-z} \right|=\frac{|z|}{1-|z|^2} \leq \frac{1}{1-|z|^2}.\]
	By the discussed above we get that $\mu(f) \geq 1/2$. On the other hand,
	\[A_f(z)=\frac{1-|z|^2}{1-z}-\frac{\overline{z}}{2},\]
	which implies $A_f(x)=1+\frac{x}{2},$ $-1<x<1,$ and in consequence $\mu(f)\leq 1/2.$ So, $\mu(f)=1/2.$
\end{example}
For the following result we need to recall the definition of Schwarzian derivative of a sense-preserving harmonic mapping $f$ in $\mathbb{D},$ which was introduced in \cite{HM} through the expression $$S_f=\partial_zP_f-\frac{1}{2}(P_f)^2,$$
where $P_f$ is given by \eqref{preschwarzian}. In \cite{ACS19} the authors define the family $NH_\lambda(K),$ $0<\lambda\leq 1$ and $K>0,$ as the set of all $K$-quasiconformal sense-preserving harmonic mappings $f$ in $\mathbb{D}$ such that
\begin{equation*}
	|S_f(z)|+\frac{|\omega'(z)|^2}{\left(1-|\omega(z)|^2\right)^2} \leq\frac{2\lambda}{\left(1-|z|^2 \right)^2 }, \qquad z\in\mathbb{D}. 
\end{equation*}
In that same paper, the authors proved the following theorem.\\

\noindent\textbf{Theorem A.} Let $\lambda\in (0,1)$ and $f\in NH_\lambda(K)$ such that $|P_f(0)|<2\sqrt{1-\lambda}$, then $f$ is bounded. The condition $|P_f(0)|<2\sqrt{1-\lambda}$ is sharp.\\\\
Using this, we easily get the following generalization of Proposition 4.1 in \cite{Po08}. 
\begin{theorem}
	Let $\lambda\in (0,1)$ and $f\in NH_\lambda(K)$, $f$ unbounded. Then $\mu(f)\geq \sqrt{1-\lambda}$.
\end{theorem}
\begin{proof}
	Let $z_0\in \mathbb{D}$ and $\sigma(z)=\dfrac{z+z_0}{1+\overline{z_0}z}$ an automorphism of $\mathbb{D}.$ Since $f\circ \sigma\in NH_\lambda(K)$ and also is unbounded, it follows from Theorem A that $|P_{f\circ \sigma}(0)|\geq2\sqrt{1-\lambda}.$ On the other hand, by direct calculation $|A_f(z_0)|=|P_{f\circ \sigma}(0)|/2.$ Thus $|A_f(z_0)|\geq\sqrt{1-\lambda}$, for $z_0\in \mathbb{D}$ arbitrary. 
\end{proof} 
In order to formulate our next result, we need to introduce, in the context of sense-preserving harmonic mappings, the concept of level sets of the density of the Poincar\'e metric. This concept was introduced by Pommerenke in \cite{Po08} in the Euclidean case and recently studied, in the hyperbolic and spherical cases, by J. Arango \textit{et al.} in \cite{AAM19} and \cite{AAM17}, respectively.

Let $f:\mathbb{D}\to \mathbb{C}$ be a sense-preserving harmonic mapping. We consider the level sets 
\[C_f(t)=\left\lbrace z\in \mathbb{D}\mid (1-|z|^2)J_f^{1/2}(z)=t\right\rbrace, \qquad \text{for}\quad 0<t<\infty.\]
Following the ideas developed in \cite{AAM19}, we assume that $A_f(z)\neq 0$ in $\mathbb{D}$ and observe that
\[\nabla\left[(1-|z|^2)J_f^{1/2}(z) \right]=2\overline{\frac{\partial}{\partial z}\left[(1-|z|^2)J_f^{1/2}(z) \right]}=\frac{2J_f^{1/2}(z)|A_f(z)|^2}{A_f(z)}.
\]
From differential equation theory (see \cite{So79}), we know that if $\delta(t),$ $0<t<\infty$, is a smooth and positive function, then the initial value problem
\begin{equation}\label{PVI1}
	\begin{cases}
		z'(t)=\dfrac{\delta(t)}{A_f(z(t))} \\
		z(t_0)=z_0
	\end{cases}
\end{equation}
has a unique maximal solution near of $z_0$, which forms an open Jordan arc orthogonal to $C_f(t)$. If we choose $\delta(t)$ such that $t=(1-|z(t)|^2)J_f^{1/2}(z(t))$, namely, the parameter of the solution $z(t)$ matches the level of the point, by (\ref{PVI1})
\begin{equation}\label{PVI2}
	\frac{1}{t}=\frac{d}{dt}\log\left[ (1-|z(t)|^2)J_f^{1/2}(z(t))\right]=2\operatorname{Re} \left[\frac{A_f(z(t))z'(t)}{1-|z|^2} \right]=\frac{2\delta(t)}{1-|z|^2},
\end{equation}
and so $\delta(t)=\dfrac{1-|z(t)|^2}{2t}$. From this we can consider instead of (\ref{PVI1}) the next initial value problem
\begin{equation}\label{PVI3}
	\begin{cases}
		z'(t)=\dfrac{1-|z(t)|^2}{2tA_f(z(t))} \\
		z(t_0)=z_0.
	\end{cases}
\end{equation}
The solutions $z(t)$, $a<t<b$, of (\ref{PVI3}) are called \textit{the trajectories} of $f$ through $z_0$. It is known that for all $z_0\in\mathbb{D},$ there is a unique trajectory through $z_0$ that goes from $\mathbb{T}=\partial\mathbb{D}$ to $\mathbb{T}.$

Using the previous development we will prove the following theorem, whose proof is similar to that of Proposition 5.1 in \cite{Po08}.
\begin{theorem}\label{thm4}
	Let $\mu(f)>0$ and $a<t_1<t_2<b$. If $z_j=z(t_j),$ $j=1,2,$ are on
	the trajectory $\Gamma$, then
	\begin{equation*}
		\frac{(1-|z_2|^2)J_f^{1/2}(z_2)}{(1-|z_1|^2)J_f^{1/2}(z_1)}\geq \operatorname{exp} \left[2\mu(f) \rho (z_1,z_2) \right].
	\end{equation*}
Here and subsequently, $\rho(z,w)$ stands for the hyperbolic distance between $z,w\in\mathbb{D}.$
\end{theorem}
\begin{proof} By (\ref{PVI2}) and (\ref{PVI3}),
	\[\log\frac{t_2}{t_1}=\int_{t_1}^{t_2}\frac{1}{t}dt=\int_{t_1}^{t_2}\frac{2|A_f(z(t))||z'(t)|}{1-|z(t)|^2}dt,\]
	from where
	\begin{equation}\label{aux1 Bloch}
		\begin{split}
			\log\frac{t_2}{t_1}&\geq2\mu(f)\int_{t_1}^{t_2}\frac{|z'(t)|}{1-|z(t)|^2}dt\\
			&=2\mu(f)l_h(\Gamma(z_1,z_2))\\
			&\geq 2\mu(f)\rho(z_1,z_2),
		\end{split}
	\end{equation}
	where $\Gamma(z_1,z_2)$ is the segment of $\Gamma$ between $z_1$ and $z_2$.
\end{proof}
For the following corollary, we recall that a harmonic mapping $f:\mathbb{D}\to\mathbb{C}$ is said to be a Bloch-type function if
\[\sup_{z\in\mathbb{D}}(1-|z|^2)\sqrt{|J_f(z)|\,}<\infty.\]
For some results on Bloch-type functions, we refer to the reader to \cite{Bloch type map}.
\begin{corollary}
	Let $f$ be a sense-preserving harmonic mapping in $\mathbb{D}.$ If $\mu(f)>0,$ then $f$ is not a Bloch-type function. 
\end{corollary}
\begin{proof}
	We assume the notation of Theorem\,\ref{thm4}. Since $\Gamma$ goes from $\mathbb{T}$ to $\mathbb{T},$ we can fix $z_1$ and let $|z_2|\to 1.$ Then from \eqref{aux1 Bloch} we obtain that $t_2\to \infty$ and consequently $b=\infty.$ Similarly we can conclude that $a=0,$ whence the trajectory $z(t)$ satisfies
	\[(1-|z(t)|^2)\sqrt{J_f(z(t))\,}=t,\qquad\qquad 0<t<\infty.\]
	Hence,
	\[\sup_{z\in\mathbb{D}}(1-|z|^2)\sqrt{J_f(z)\,}=\infty,\]
	which ends the proof.
\end{proof}

\section{Linear invariance order for harmonic mappings}\label{section linear invariance}

Given a sense-preserving harmonic mapping $f=h+\bar{g}$ in $\mathbb{D},$ \textit{the upper order} (or simply the order) of $f$ is defined by
\[\left\| A_f \right\|:=\sup_{z\in\mathbb{D}}|A_f(z)|.\]
Note that by Proposition\,\ref{prop properties of A_f} it follows easily that for all affine mapping $L$ and $\sigma\in \text{Aut}\left(\mathbb{D}\right),$
\[A_{L\circ f\circ\sigma}(z)= \frac{\sigma'(z)}{|\sigma'(z)|}A_f(\sigma(z)),\qquad\text{for all}\quad z\in\mathbb{D},\]
which implies $\left\Vert A_{L\circ f\circ\sigma}\right\Vert =\left\Vert A_{f}\right\Vert.$

It is remarkable that the order of $f$ coincides with the specified order of the affine and linear invariant family generated by $f$ (the linear-affine hull of $\left\lbrace f \right\rbrace $), which was defined in \cite{specified order Graf}. It also coincides with the \textit{strong order} studied in \cite{old and new order ALIF harmonic}, see also \cite{Liu ponnusamy}.
\begin{remark}
	As we have seen, in the setting of harmonic mappings $A_f$ has similar properties to the operator defined by Pommerenke in \eqref{Af1}, but there are some striking differences as well. For example, a classical result in geometric function theory states that for an analytic function $h,$ $\left\Vert A_{h}\right\Vert =1$ exactly if $h$ is a convex univalent function. In the context of harmonic mappings, it is known (see \cite{Clunie and Sheil}) that if $f$ is a convex sense-preserving harmonic mapping in $\mathbb{D}$, then $\left\|A_f \right\|\leq 3/2,$ the function $L$ of Example\,\ref{example L} shows that the bound is sharp. However, the following example shows that in the setting of harmonic mappings, the condition $\left\Vert A_{f}\right\Vert =3/2$ does not imply that $f$ is convex.   
\end{remark}
\begin{example}\label{ex f is not convex}
Given $n\geq 2,$ we consider the harmonic mapping $f(z)=z+\dfrac{1}{n}\overline{z}^{n},$ which is univalent in $\mathbb{D}$ and $f\left(  \mathbb{D}\right)$ is not convex \cite[p. 3]{Du04}.  Moreover $\left\Vert A_{f}\right\Vert =3/2.$ Indeed, by direct calculation 
	\[
	A_{f}(z)=-\overline{z}\left[  1+\frac{n-1}{2}\frac{\left\vert
		z\right\vert ^{2\left(  n-2\right)  }}{1+\left\vert z\right\vert
		^{2}+\left\vert z\right\vert ^{4}+\cdots+\left\vert z\right\vert ^{2\left(
			n-2\right)  }}\right],
	\]
	and, in consequence
	\[
	\left\vert A_{f}(z) \right\vert =\left\vert z\right\vert
	+\frac{n-1}{2}\frac{\left\vert z\right\vert ^{2n-3}}{1+\left\vert z\right\vert
		^{2}+\left\vert z\right\vert ^{4}+\cdots+\left\vert z\right\vert ^{2\left(
			n-2\right)  }}.
	\]
	Since the function
	\[
	g(x)=x+\dfrac{n-1}{2}\dfrac{x^{2n-3}}{1+x^{2}+\left(
		x^{2}\right)  ^{2}+\cdots+\left(  x^{2}\right)  ^{\left(  n-2\right)  }}%
	\]
	is increasing on $[0,1]$, it follows that $\left\Vert A_{f}\right\Vert =g(1)=3/2.$ 
\end{example}
At present the order $\left\|A_f \right\| $ is only known for very few types of
harmonic mappings. In the following proposition, the order of a stable harmonic convex mapping is established. Let $f=h+\overline{g}$ be a sense-preserving harmonic mapping in $\mathbb{D}.$ In \cite{HeMa13} the authors define that $f$ is stable harmonic convex (SHC) if all the functions $f_{\lambda} = h + \lambda g$ with $|\lambda| = 1$ are convex in $\mathbb{D}$.
\begin{proposition}\label{ord}
	If $f=h+\overline{g}$ is a sense-preserving harmonic mapping in $\mathbb{D}$, then $\left\Vert A_{f}\right\Vert \geq1.$ Also, if $f$ is SHC in $\mathbb{D}$, then $\left\Vert A_{f}\right\Vert =1.$
\end{proposition}
\begin{proof} For the proof of $\left\Vert A_{f}\right\Vert \geq1,$ see \cite[Corollary 3.5]{old and new order ALIF harmonic}. The other part of the proof can be done by following a similar argument to that used to obtain \eqref{eq Af and a2}; however we provide an alternative proof, which uses the same ideas as in the proof of Proposition\,\ref{prop stable concave} and it will be used later in the proof of the next corollary.
	
As in the proof of Proposition\,\ref{prop stable concave}, we have
\begin{equation}\label{eq aux prop 3}
A_f(z)+\sigma(\lambda)\frac{(1-|z|^2)\omega'(z)}{2(1-|\omega(z)|^2)}=A_{f_\lambda}(z),
\end{equation}
for all $z\in\mathbb{D}$ and $|\lambda|=1,$ where $\sigma$ is defined as in \eqref{eq def auth sigma}. Thus, given $z_0\in\mathbb{D},$ we conclude of the condition $\omega(z_0)\in\mathbb{D}$ and the fact that $\varphi$ maps $\partial\mathbb{D}$ onto $\partial\mathbb{D},$ that there is $\lambda:=\lambda(z_0)$ such that
\[\left|A_f(z_0)+\sigma(\lambda)\frac{(1-|z_0|^2)\omega'(z_0)}{2(1-|\omega(z_0)|^2)}\right|=\left| A_f(z_0) \right|+\left| \frac{(1-|z_0|^2)\omega'(z_0)}{2(1-|\omega(z_0)|^2)} \right|. \]
It follows from \eqref{eq aux prop 3} and $\left\| A_{f_\lambda} \right\|\leq 1$ that
\begin{equation}\label{eq aux 2 shc}
\left| A_f(z_0) \right|+\left| \frac{(1-|z_0|^2)\omega'(z_0)}{2(1-|\omega(z_0)|^2)} \right|\leq 1,
\end{equation}
for all $z_0\in\mathbb{D}.$ Therefore $\left\|A_f \right\|\leq 1,$ which implies $\left\|A_f \right\|= 1.$
\end{proof}
\begin{corollary}
There is not a SHC mapping with dilatation satisfying
\[\inf_{z\in \mathbb{D}}\left\lbrace\frac{(1-|z|^2)|\omega'(z)|}{1-|\omega(z)|^2} \right\rbrace >0. \]
In particular, there is not a SHC mapping with dilatation $\omega(z)=z,$ $z\in\mathbb{D}.$
\end{corollary}
\begin{proof}
The proof follows immediately from \eqref{eq aux 2 shc} and the fact that $\left\| A_f\right\|\geq 1,$ for every sense-preserving harmonic mapping $f$ defined in $\mathbb{D}.$
\end{proof}
Next, we will prove a result which establishes a relation between the order of a sense-preserving harmonic mapping and linearly connected domains.

We recall that a domain $\Omega\subset\mathbb{C}$ is {\it linearly connected} if there exists a constant $1\leq M<\infty$ such that any two points $w_1,w_2\in\Omega$ are joined by a path $\gamma\subset\Omega$ of length $\ell(\gamma)\leq M|w_1-w_2|$. In such case, we say that $\Omega$ is $M-$linearly connected. For piecewise smoothly bounded domains,
linear connectivity is equivalent to the boundary having no inward-pointing cusps. It is clear that if $\Omega$ is $M-$linearly connected, then so is $c\Omega=\left\lbrace cz\mid z\in\Omega \right\rbrace,$ for all $c\in\mathbb{C}.$ Moreover, for all $\alpha\in\mathbb{D},$ if $L(z)=z+\alpha\overline{z}$ then $L(\Omega)$ is $b-$linearly connected, where $b=M\frac{1+|\alpha|}{1-|\alpha|}.$

Pommerenke \cite[Theorem 5.7]{P92} proved that if $f$ maps $\mathbb{D}$ conformally onto a linearly connected domain, then $\left\Vert A_{f}\right\Vert<2.$ The hyperbolic version of this result was later presented in \cite[Theorem 6]{hyperbolic order}. In the setting of harmonic mappings we have the following result.
\begin{theorem}
	Let $\Omega\subset\mathbb{C}$ be a $M-$linearly connected domain. There is $0<c<1$ such that if a univalent harmonic mapping $f$ satisfies the conditions $f(\mathbb{D})=\Omega$ and $|\omega(z)|\leq c$ for all $z\in\mathbb{D},$ then $\left\Vert A_{f}\right\Vert<2.$ 
\end{theorem}
\begin{proof}
We will prove that $0<c\leq 1/(7+2M)$ satisfies the requirement of the theorem. Under this condition on $c,$ which guarantees that $0<c<1/(1+M),$ it is shown in \cite[Theorem\,2]{linealy connected Martin-Rodrigo} that $|\omega(z)|\leq c$ implies $h$ univalent. So proceeding as in the proof of \eqref{eq Af and a2}, we get $|A_f(z)|\leq 2$ for all $z\in\mathbb{D}.$ Thus, if we
	suppose that $\left\Vert A_{f}\right\Vert=2,$ we can then choose a sequence $(z_n)\subset\mathbb{D}$ such that $|A_f(z_n)|$ tends to 2 as $n$ tends to infinity. Next, we consider the sequences of functions in $S_H$ and $S_H^0,$ given respectively by
	\[f_n(z)=\frac{f\left(\frac{z+z_n}{1+\overline{z_n}z}\right)-f(z_n) }{(1-|z_n|^2)h'(z_n)}=:h_n(z)+\overline{g_n(z)}\]
	and
	\[F_n(z)=\frac{f_n(z)-\overline{b_1(n)f_n(z)}}{1-|b_1(n)|^2}=:H_n(z)+\overline{G_n(z)},\]
	where $b_1(n)=g_n'(0)=\omega_{f_n}(0)=\omega(z_n).$ From what was discussed above, it can be concluded that $F_n(\mathbb{D})$ is a $b-$linearly connected domain, with $b=2+M.$ Moreover, if $W_n$ and $\omega_n$ denote the dilatation of $F_n$ and $f_n$ respectively, we obtain
	\[W_n(z)=\frac{G_n'(z)}{H_n'(z)}=\frac{g_n'(z)-b_1(n)h_n'(z)}{h_n'(z)-\overline{b_1(n)}g_n'(z)}=\frac{\omega_n(z)-b_1(n)}{1-\overline{b_1(n)}\omega_n(z)}.\]
	It follows from the condition on $c$ and $\left\|\omega \right\|\leq c$ that $\left\|W_n \right\|<1/(1+b),$ for all $n.$ We conclude from the remark after Theorem\,2 in \cite{linealy connected Martin-Rodrigo} that $H_n(\mathbb{D})$ is $\rho-$linearly connected, for some constant $\rho>1$ independent on $n.$ 
	
	On the other hand, given that $(H_n)$ is a sequence of conformal mappings in the unit disk with $H_n(0)=0$ and $H_n'(0)=1$ for all $n,$ we can suppose, without loss of generality, that $(H_n)$ converges locally uniformly in $\mathbb{D}$ to a univalent function $H:\mathbb{D}\to\mathbb{C},$ satisfying $H(0)=0$ and $H'(0)=1.$ It can be shown that $H(\mathbb{D})$ also is linearly connected. Indeed, given $w,\tilde{w}\in H(\mathbb{D}),$ there are $z,\tilde{z}\in\mathbb{D}$ such that $H(z)=w$ and $H(\tilde{z})=\tilde{w}.$  Thus, for all $n$ there is a curve $\gamma_n\subset\mathbb{D},$ with endpoints $z,\tilde{z},$ satisfying
	\[\ell(H_n(\gamma_n))\leq\rho |H_n(z)-H_n(\tilde{z})|.\]
	It follows from a result of Gehring and Hayman \cite{Gehring Haymann} (see also \cite[p. 88]{P92}) that there exists an absolute constant $C$ such that
	\[\ell(H_n(S))\leq C |H_n(z)-H_n(\tilde{z})|,\quad\text{ for all }\quad n, \]
	where $S$ is the hyperbolic segment with endpoints $z,\tilde{z}.$ By letting $n\to\infty,$ we conclude that
	\[\ell(H(S))\leq C |H(z)-H(\tilde{z})|=C |w-\tilde{w}|,\]
	whence $H(\mathbb{D})$ is a $C-$linearly connected domain. This contradicts the fact that $H$ is a rotation of the Koebe function $k(z)=z/(1-z)^2,$ which is a consequence of
	\[\left| \frac{H_n''(0)}{2}\right|=|A_f(z_n)|\to 2\quad \text{ as }\quad n\to\infty.\]
	Thus, we can conclude that $\left\| A_f\right\| <2.$
\end{proof}
We finish with some remarks (see Proposition\,\ref{prop case equality} and subsequent remark) about a distortion theorem for harmonic mappings with finite order, which can be found, with some changes in its presentation, in \cite{old and new order ALIF harmonic}. See also \cite{Liu ponnusamy,Schaubroeck bounds jacobian,Sheil small} for related results. We will present the proof since its argument will be used to analyze the case when there is equality in \eqref{desighiper} and the proof of the part ``only if'' is slightly different to that of \cite{old and new order ALIF harmonic}.
\begin{theorem}\label{thm distortion}
Let $f=h+\overline{g}$ be a sense-preserving harmonic mapping in $\mathbb{D}$ and $\alpha \geq 0.$ Then $\left\vert A_{f}\left(z\right) \right\vert \leq \alpha$ for all $z\in \mathbb{D}$ if and only if 
\begin{equation}\label{desighiper}
		\exp \left( -2\alpha \rho\left( z_{0},z_{1}\right) \right) \leq \frac{%
			\left( 1-\left\vert z_{1}\right\vert ^{2}\right) J_f^{1/2}\left( z_{1}\right) }{%
			\left( 1-\left\vert z_{0}\right\vert ^{2}\right) J_f^{1/2}\left( z_{0}\right) }%
		\leq \exp \left( 2\alpha \rho\left( z_{0},z_{1}\right) \right),
\end{equation}
for all $z_{0},z_{1}\in \mathbb{D}.$ If equality holds in any of these inequalities for $z_0,z_1\in\mathbb{D}$, $z_0 \neq z_1$, then $\left\vert A_{f}\left(z\right) \right\vert =\alpha ,$ for all $z$ in the hyperbolic segment $S:=S(z_0,z_1)$ with extremes $z_0$ and $z_1$ and we get equality in the corresponding side of (\ref{desighiper}) for all $\widehat{z_0}, \widehat{z_1}$ in $S$. 
\end{theorem}
\begin{proof}
Let $\gamma $ be the hyperbolic segment joining $z_{0}$	and $z_{1}$ in $\mathbb{D}$ and let $z=z\left( s\right) ,$ $0\leq s\leq L:=l_{h}\left(\gamma\right),$ be a parametrization of $\gamma $ by hyperbolic arc length. Then $z^{\prime }=\left( 1-\left\vert z\right\vert
	^{2}\right) e^{i\theta \left( s\right) },$ from which we get	
\begin{align*} 
		\frac{d}{ds}\log \left( 1-\left\vert z\right\vert ^{2}\right) J_f^{1/2} &  = -\frac{2\mathrm{Re} \left\lbrace z'\overline{z}\right\rbrace }{1-\left\vert
			z\right\vert ^{2}} + \frac{1}{2}\frac{2\mathrm{Re} \left\lbrace \left( h''\overline{h'}+g''\overline{g'}\right)z' \right\rbrace }{J_f}
		\\
		& = 2\mathrm{Re} \left\{ -\overline{z}\frac{ z'}{1-\left\vert z\right\vert ^{2}} + \frac{1}{2}\frac{ \left( h''\overline{h'}+g''\overline{g'}\right)z' }{J_f}\right\rbrace \\
		& =2\mathrm{Re} \left\lbrace\left(-\overline{z}+\frac{1-|z|^2}{2}P_f \right)e^{i\theta(s)}  \right\rbrace
\end{align*}
and, in consequence,
\begin{equation}\label{TH1}
\frac{d}{ds}\log \left( 1-\left\vert z\right\vert ^{2}\right) J_f^{1/2}(z) =  2\mathrm{Re} \left\{ A_{f}\left( z\right) e^{i\theta(s)}\right\}.
\end{equation}%
It follows by hypothesis that 
\begin{equation*}
		-2\alpha \leq \frac{d}{ds}\log \left( 1-\left\vert z\right\vert ^{2}\right)
		J_f^{1/2}\left( z\right) \leq 2\alpha,
\end{equation*}
whence we get, after integration with respect to $s$ in the interval $\left[ 0,L\right],$
\begin{equation*}
		-2\alpha L\leq \log \frac{\left( 1-\left\vert z_{1}\right\vert ^{2}\right)
			J_f^{1/2}\left( z_{1}\right) }{\left( 1-\left\vert z_{0}\right\vert ^{2}\right)
			J_f^{1/2}\left( z_{0}\right) }\leq 2\alpha L,
\end{equation*}
which is equivalent to (\ref{desighiper}).
	
For the converse, we set $u(z)=\log \left( \left(1-|z|^2\right) J_f^{1/2}(z) \right) $ for all $z\in \mathbb{D}.$ For $z_{0}\in \mathbb{D}$ arbitrary, we obtain from (\ref{desighiper})
\begin{equation*}
		\left\vert u(z) -u( z_{0}) \right\vert =\left\vert
	\log \frac{\left( 1-|z|^{2}\right)J_f^{1/2}(z)  }{\left( 1-|z_{0}|^{2}\right)J_f^{1/2}(z_{0})  }\right\vert \leq 2\alpha \rho\left( z_{0},z\right)
\end{equation*}
and therefore, 
\begin{equation*}
		\dfrac{\left\vert u\left( z\right) -u\left( z_{0}\right) \right\vert }{%
			\left\vert z-z_{0}\right\vert }\leq 2\alpha \dfrac{\rho\left(
			z_{0},z\right) }{\left\vert z-z_{0}\right\vert }.
\end{equation*}%
Now let $z$ approach $z_{0}$ in the direction of maximum growth of $u$ at $z_{0}$, to get
	\begin{equation*}
		2\left|\frac{\partial u}{\partial z}(z_0) \right|=\left\vert \nabla
		u\left( z_{0}\right) \right\vert=\lim_{z\rightarrow z_{0}}\dfrac{\left\vert u\left( z\right) -u\left(
			z_{0}\right) \right\vert }{\left\vert z-z_{0}\right\vert }\leq 2\alpha \lambda _{\mathbb{D}}\left(
		z_{0}\right), 
	\end{equation*}
where $\lambda _{\mathbb{D}}$ denotes the Poincar\'e density of the unit disk. It follows from here and from \eqref{diferential form of A_f} that $|A_{f}(z_0)|\leq \alpha.$

Next we consider the case of equality. Without loss of generality, we assume that there is equality on the right side of (\ref{desighiper}) in $z_0,z_1$, $z_0 \neq z_1$. That is
	\[\log\frac{%
		\left( 1-\left\vert z_{1}\right\vert ^{2}\right) J_f^{1/2}\left( z_{1}\right) }{%
		\left( 1-\left\vert z_{0}\right\vert ^{2}\right) J_f^{1/2}\left( z_{0}\right) }%
	=  2\alpha L, 
	\]
which implies
\[	\int_0^L \frac{d}{ds}\log \left( 1-\left\vert z(s)\right\vert ^{2}\right) J_f^{1/2}(z(s)) ds=  2\alpha L. \]
From here and 
\[\frac{d}{ds}\log \left( 1-\left\vert z(s)\right\vert ^{2}\right) J_f^{1/2}(z(s))\leq 2\alpha , \;\;\; \text{for all}  \;z\in S, \]
we obtain from (\ref{TH1}) that 
\begin{equation}\label{TH1.1}
2\mathrm{Re} \left\{ A_{f}\left( z\right) e^{i\theta(s)}\right\}=\frac{d}{ds}\log \left( 1-\left\vert z(s)\right\vert ^{2}\right) J_f^{1/2}(z(s)) = 2\alpha , \;\;\; \text{for all}  \;z\in S.
\end{equation}
It follows that $|A_f(z)|=\alpha$ for all $z\in S.$ The same integration argument just given allows us to conclude from (\ref{TH1.1}) that
\[\log\frac{%
	\left( 1-\left\vert \widehat{z_1}\right\vert ^{2}\right) J_f^{1/2}\left(\widehat{z_1}\right) }{%
	\left( 1-\left\vert \widehat{z_0}\right\vert ^{2}\right) J_f^{1/2}\left( \widehat{z_0}\right) }%
=  2\alpha \rho(\widehat{z_0}, \widehat{z_1}), 
\]
for all $\widehat{z_0}, \widehat{z_1}$ in $S$.
\end{proof}
\begin{remark}\label{remark 2}
If we suppose $z_0=0$ then the inequality (\ref{desighiper}) becomes
\begin{equation}\label{J_f}
\frac{\left(1-|z|\right)^{2\alpha-2}}{\left(1+|z|\right)^{2\alpha+2}}\leq\frac{J_{f}(z)}{J_{f}(0)} \leq\frac{\left(1+|z|\right)^{2\alpha-2}}{\left(1-|z|\right)^{2\alpha+2}},
\end{equation}
for all $z\in \mathbb{D},$ which coincides with (2.1) in \cite{old and new order ALIF harmonic}. In that same paper, the authors mention that the equality is achieved by a certain affine mapping $f_{\alpha}$ of the function
\[k_{\alpha}(z)=\frac{1}{2\alpha}\left[ \left( \frac{1+z}{1-z}\right)^{\alpha}-1 \right],\qquad z\in\mathbb{D},\]
which follows by direct calculations.
\end{remark}
We will obtain a condition under which a harmonic function, with dilatation having certain property, satisfies the equality in one of the inequalities of \eqref{J_f}. As a particular case, it is shown that $f_{\alpha}$ is essentially the only harmonic mapping with constant dilatation satisfying the equality in one of the inequalities of \eqref{J_f}.

\begin{proposition}\label{prop case equality}
Let $f=h+\bar{g}$ be a sense-preserving harmonic mapping in $\mathbb{D},$ normalized under the condition $J_f(0)=1,$ and we suppose that $|A_f(z)|\leq \alpha$ for all $z\in\mathbb{D}.$\\

(i)  If there exists $z_0=re^{i\theta}\neq 0$ satisfying the equality in the right hand side of (\ref{J_f}) and either $\omega(te^{i\theta})\in\mathbb{R},$ for $0\leq t<1,$ or $\omega$ is constant, then $h$ satisfies the equation
\begin{equation}\label{h}
\frac{h''(z)}{h'(z)}=\frac{\omega(z)\omega'(z)}{1-\omega^2(z)}+2e^{-i\theta}\frac{\alpha+e^{-i\theta}z}{1-(e^{-i\theta}z)^2}, \qquad z\in \mathbb{D}.
\end{equation}

(ii)  If there exists $z_0=re^{i\theta}\neq 0$ satisfying the equality in the left hand side of (\ref{J_f}) and either $\omega(te^{i\theta})\in\mathbb{R},$ for $0\leq t<1,$ or $\omega$ is constant, then $h$ satisfies the equation
\begin{equation}\label{eq left h}
\frac{h''(z)}{h'(z)}=\frac{\omega(z)\omega'(z)}{1-\omega^2(z)}-2e^{-i\theta}\frac{\alpha-e^{-i\theta}z}{1-(e^{-i\theta}z)^2}, \qquad z\in \mathbb{D}.
\end{equation}
\end{proposition} 
\begin{proof}
Since the proofs of \eqref{h} and \eqref{eq left h} are similar, we only prove \eqref{h}. By defining $\varphi(z)=f(e^{i\theta}z),$ we have $\left\vert A_{\varphi}\left(z\right) \right\vert \leq \alpha,$ $\omega_{\varphi}(z)=\omega_f(e^{i\theta}z),$ and $J_{\varphi}(z)=J_{f}(e^{i\theta}z)$ for all $z\in \mathbb{D}$. Then, without loss of generality, we can assume that
\[J_{f}(r)=\frac{\left(1+r\right)^{2\alpha-2}}{\left(1-r\right)^{2\alpha+2}}\]
and either $\omega(t)\in\mathbb{R},$ for $0\leq t<r,$ or $\omega$ is constant. So, proceeding as in the proof of \eqref{TH1.1}, we find that $\text{Re}\left\lbrace A_f(t) \right\rbrace =\alpha,$ for all $0\leq t\leq r.$ Consequently, the hypothesis $\left\| A_f \right\|\leq \alpha $ implies $A_f(t)=\alpha,$ for all $0\leq t\leq r,$ or equivalently,
\[ A_h(t)-\frac{1-t^2}{2}\frac{\overline{\omega(t)}\omega'(t)}{1-|\omega(t)|^2}=\alpha ,\qquad 0\leq t\leq r. \]
Then it follows from the conditions on $\omega$ and \eqref{Af1} that
\[ \frac{1-t^2}{2}\frac{h''(t)}{h'(t)}-t-\frac{1-t^2}{2}\frac{\omega(t)\omega'(t)}{1-\omega^2(t)}=\alpha ,\qquad 0\leq t\leq r.\]
Hence, by analytic continuation we obtain that
\[ \frac{1-z^2}{2}\frac{h''(z)}{h'(z)}-z-\frac{1-z^2}{2}\frac{\omega(z)\omega'(z)}{1-\omega^2(z)}=\alpha ,\qquad z\in\mathbb{D},
\]
which proves $(i).$
\end{proof}
\begin{remark}
The above reasoning leads us to the following conclusions:\\

(a) If $\omega$ is constant and $f=h+\overline{g}$ satisfies equality on the right hand side of (\ref{J_f}) at some $z_0\in(0,1),$ then by (\ref{h}), 
	\[\frac{h''(z)}{h'(z)}=2\frac{\alpha+z}{1-z^2}=2\left[ \frac{\alpha+1}{2}\frac{1}{1-z}+\frac{\alpha-1}{2}\frac{1}{1+z}\right], \]
from where, assuming that $h'(0)=1$,
	\[\log h'(z)=(\alpha-1)\log (1+z)-(\alpha+1)\log (1-z).\]
Therefore,
\[h'(z)=\frac{(1+z)^{\alpha-1}}{(1-z)^{\alpha+1}}.\]
Now, if we suppose that $h(0)=0,$ we easily see that 
	\[h(z)=\frac{1}{2\alpha}\left[\left( \frac{1+z}{1-z}\right)^\alpha-1 \right]=:k_\alpha(z).\]
Thus, the harmonic mapping $f=h+\overline{g},$ with constant dilatation and normalized under the conditions $h(0)=g(0)=0$ and $h'(0)=1,$ satisfying the equality on the right hand side of (\ref{J_f}) at some $z_0\in(0,1),$ is given by $f(z)=k_\alpha(z)+\overline{\omega k_\alpha(z)}.$ Moreover, we can see in general that if $f$ satisfies equality on the right hand side of (\ref{J_f}) at some $z_0\in \mathbb{D}\setminus\left\lbrace 0\right\rbrace $, then $f(z)=k_{\alpha,\theta}(z)+\overline{\omega k_{\alpha,\theta}(z)},$ where $k_{\alpha,\theta}$ is some rotation of $k_\alpha$, with $\theta$ depending on $z_0.$\\

(b) The same conclusion as in (a) is obtained if we assume that $f,$ with the above properties, satisfies equality on the left hand side of (\ref{J_f}) at some $z_0\in\mathbb{D}.$\\

(c) Let $L$ be as in Example\,\ref{example L}. Then, a straightforward calculation gives
\[\omega_L(z):=\omega(z)=-z\qquad\text{and}\qquad  J_L(z)=\frac{1-|z|^2}{|1-z|^6},\quad z\in\mathbb{D}.\]
Therefore, with $\alpha=3/2,$ $L$ satisfies equality on the left hand side of (\ref{J_f}), for all $z\in(-1,0).$ We prove that if a sense-preserving harmonic mapping $f=h+\overline{g},$ with $h(0)=g(0)=0,$ $h'(0)=1,$ and $\omega(z)=-z,$ satisfies equality on the left hand side of (\ref{J_f}) at some $z_0\in(-1,0),$ then $f=L.$ Indeed, if such $f$ exists, we must have by \eqref{eq left h} that
\[\frac{h''(z)}{h'(z)}=\frac{z}{1-z^2}+2\frac{\alpha+z}{1-z^2},\]
where we have used the condition $\theta=\pi.$ From here, if $\alpha=3/2,$ we obtain by integration and $g'=\omega h'$ that 
\[h'(z)=\frac{1}{(1-z)^3}\qquad\text{ and }\qquad g'(z)=-\frac{z}{(1-z)^3},\]
whence $f(z)=L(z)$ for all $z\in \mathbb{D}.$
\end{remark}

\end{document}